\documentclass[a4paper,12pt,draft]{amsart}

\usepackage[margin=30mm]{geometry}
\usepackage{amssymb, amsmath, amsthm, amscd}

\usepackage[T1]{fontenc}
\usepackage{eucal,mathrsfs,dsfont}
\usepackage{color}
\usepackage{verbatim}

\definecolor{mno}{rgb}{0.5,0.1,0.5}

\renewcommand{\leq}{\leqslant}
\renewcommand{\geq}{\geqslant}
\renewcommand{\le}{\leqslant}
\renewcommand{\ge}{\geqslant}

\newcommand{\R}{\mathds R}
\newcommand{\B}{\mathds B}
\newcommand{\Hh}{\mathds H}
\newcommand{\Uu}{\mathds U}

\newcommand{\Pp}{\mathds P}
\newcommand{\Ee}{\mathds E}
\newcommand{\I}{\mathds 1}

\newcommand{\N}{\mathds N}
\newcommand{\var}{\mathrm{Var}}
\newcommand{\Bb}{\mathscr{B}}
\newcommand{\scalp}[2]{\langle#1,#2\rangle}

\newcommand{\supp}{\operatorname{supp}}
\newcommand{\Rank}{\operatorname{Rank}}
\newcommand{\Leb}{\operatorname{Leb}}

\newcommand{\nnu}{\mathsf{n}}

\newtheorem{theorem}{Theorem}[section]

\newtheorem{proposition}[theorem]{Proposition}
\newtheorem{corollary}[theorem]{Corollary}

\theoremstyle{definition}

\newtheorem{example}[theorem]{Example}
\newtheorem{remark}[theorem]{Remark}
\newtheorem*{ack}{Acknowledgement}

\begin{document}
\allowdisplaybreaks
\title[Linear evolution equations with cylindrical
L\'{e}vy noise]{\bfseries Linear Evolution Equations with
Cylindrical L\'{e}vy Noise: Gradient
Estimates and Exponential Ergodicity}

\author{Jian Wang}

\thanks{
\emph{J.\ Wang:}
School of Mathematics and Computer Science, Fujian Normal
University, 350007, Fuzhou, P.R. China.
\texttt{jianwang@fjnu.edu.cn}}

\date{}

\maketitle

\begin{abstract} Explicit coupling property and gradient estimates are investigated for the linear
evolution equations on Hilbert spaces driven by an additive
cylindrical L\'{e}vy process. The results are efficiently applied to establish the exponential
ergodicity for the associated transition semigroups. In particular, our
results extend recent developments on related topic for
cylindrical symmetric $\alpha$-stable processes.

\medskip

\noindent\textbf{Keywords:} Cylindrical L\'{e}vy processes; linear
evolution equations; coupling property; gradient estimates;
exponential ergodicity

\medskip

\noindent \textbf{MSC 2010:} 60H15; 60J75; 60G51; 35R60.
\end{abstract}

\section{Introduction and Main Results}\label{section1}
Let $\Hh$ be a real separable Hilbert space with the norm $\|\cdot\|$, and $(A, D(A))$ be a
linear possibly unbounded operator which generates a $C_0$-semigroup
(i.e.\ a strongly continuous one-parameter semigroup of linear
operators) $(T_t)_{t\ge0}$ on $\Hh$. Consider the following linear
evolution equation:
\begin{equation}\label{eq100} dX_t=AX_t\,dt+ dZ_t,\quad t\ge0,\quad X_0=x\in
\Hh,\end{equation} where $Z=(Z_t)_{t\ge0}$ is an infinite
dimensional L\'{e}vy process which may take values in a Hilbert
space $\Uu$ usually greater than $\Hh$. The Markov process
$X=(X_t^x)_{t\ge0}$ determined by \eqref{eq100} (if exists) is
called a L\'{e}vy driven Ornstein-Uhlenbeck process, e.g.\ see
\cite{App,App1,Ch,PZ, Ri} and the references therein. The associated
transition semigroup acting on ${B}_b(\Hh)$, the class of all
bounded measurable functions on $\Hh$, is given by
\begin{equation*} \label{eq4} P_tf(x):=\Ee f(X_t^x),\end{equation*}
which is called a generalized Mehler semigroup in the sense of \cite{BR,FR,LR}.

Throughout this paper we suppose that the infinite dimensional L\'{e}vy process
$Z=(Z_t)_{t\ge0}$ in \eqref{eq100} is defined by the orthogonal expression
\begin{equation} \label{eq2} Z_t=\sum_{n=1}^\infty Z_t^ne_n,\quad t\ge0,\end{equation}
where $(e_n)$ is an orthonormal basis of $\Hh$ and
$(Z_t^n)_{t\ge0,n\ge1}$ are independent real valued pure jump
L\'{e}vy processes defined on a fixed stochastic basis. We call
$(Z_t)_{t\ge0}$ given by \eqref{eq2} an additive cylindrical
L\'{e}vy process, see \cite{AM} for the recent study on cylindrical L\'{e}vy processes. We also need the following
assumption on the operator $(A,D(A))$:

\indent\paragraph{\bf{{Hypothesis (H)}}}
 \emph{ The operator $(A, D(A))$ is a self-adjoint operator such that for the
  fixed basis $(e_n)$ in \eqref{eq2} it verifies that $(e_n)\subset D(A)$, $Ae_n=-\gamma_ne_n$ with $\gamma_n>0$
  for any $n\ge1$, and $\lim\limits_{n\to\infty}\gamma_n=\infty.$ }

That is, we require that the basis $(e_n)$ from the representation
\eqref{eq2} are eigenvectors of $A$. Then, we can identify $\Hh$
with $l^2=\{x=(x_n): \sum_{n=1}^\infty x_n^2<\infty\},$ and $A$ with
the diagonal operator $Ax_n=-\gamma_n x_n$ for $(x_n)\in l^2$ and
$n\ge1$. Thus, under assumptions \eqref{eq2} and {\bf (H)}, the
equation \eqref{eq100} can be solved for each coordinate separately,
i.e.,
$$dX_t^{x,n}=-\gamma_nX_t^{x,n}\,dt+dZ_t^n,\quad X_0^{x,n}=x_n,\,\, n\in \N,$$ with $x=(x_n)\in l^2$. Therefore, the unique solution to the equation \eqref{eq100} can be seen as a
stochastic process $X=(X_t^x)_{t\ge0}$ taking values in
$\R^\N$ with components
\begin{equation}\label{th2000} X_t^{x,n}=e^{-\gamma_nt}x_n+ \int_0^te^{-\gamma_n(t-s)}\,dZ_s^n,\quad n\in \N, t\ge0.\end{equation}

In the following, we always assume that the process
$X=(X_t^x)_{t\ge0}$ takes values in $\Hh$. According to Corollary \ref{thex1} below, this is guaranteed if and only if  \begin{equation}\label{not2}\sum_{n=1}^\infty \int_0^1\bigg(e^{-2\gamma_n s}\int_{\{|z|\le e^{\gamma_ns}\}}z^2\,\nu_n(dz)+ \int_{\{|z|>e^{\gamma_ns}\}}\,\nu_n(dz)\bigg)\,ds<\infty,
\end{equation} where for any $n\ge1$, $\nu_n$ is the L\'{e}vy measure of the
L\'{e}vy process $(Z_t^n)_{t\ge0}$ on $\R$.
 We first study the coupling property and explicit gradient
estimates of the associated Markov semigroup for the process $X$. Recall that the process $X$ has
successful couplings (or has the coupling property) if and only if
for any $x,y\in\Hh$,
\begin{equation*}\label{prex1}
    \lim_{t\rightarrow\infty}\|P_t(x,\cdot)-P_t(y,\cdot)\|_{\var}=0,
\end{equation*}
where $P_t(x,dz)$ is the transition kernel of the process $X$ and
$\|\cdot\|_{\var}$ stands for the total variation norm. Let $(P_t)_{\ge0}$ be the transition semigroup of the  process $X$. The uniform norm of its gradient is defined as follows:
 $$\|\nabla P_t\|_\infty:=\sup\bigg\{ |\nabla_zP_tg(x)|: \|z\|\le 1, x\in \Hh, g\in B_b(\Hh)\textrm{ with } \|g\|_{\Hh,\infty}\le1\bigg\},$$ where
$$|\nabla_zP_tg(x)|:=\limsup_{\varepsilon\to 0}\frac{1}{\varepsilon}\Big|P_tg(x+\varepsilon z)-P_tg(x)\Big|,$$ and
$\|g\|_{\Hh,\infty}$ is denoted by the supremun norm, i.e.\
$\|g\|_{\Hh, \infty}=\sup_{x\in \Hh}|g(x)|.$ A continuous function $f:[0,\infty)\to [0,\infty)$ is said to be a Bernstein function if $(-1)^{k} f^{(k)}(x)\le 0$ for all $x>0$ and $k\ge 1$. One of our main contributions is as follows.

\begin{theorem}\label{th2} Let $X=(X_t^x)_{t\ge0}$ be the process with components \eqref{th2000} taking values in $\Hh$, i.e.\ \eqref{not2} holds. Suppose that for any $n\ge1$,
the L\'{e}vy measure $\nu_n$ of the
L\'{e}vy
process $(Z_t^n)_{t\ge0}$ on $\R$ satisfies that
\begin{equation}\label{th21}\nu_n(dz)\ge |z|^{-1}f_n(z^{-2})\,dz,\end{equation} where
$f_n$ is a Bernstein function with $f_n(0)=0$ and $\lim\limits_{r\to
\infty}\frac{f_n(r)}{\log (1+r)}=\infty.$ Then, for any $t>0$ and
$x,y\in\Hh$,
\begin{equation}\label{th22}\|\nabla P_t\|_\infty\le  C_t,\end{equation} and
$$\|P_t(x,\cdot)-P_t(y,\cdot)\|_{\var}\le 2C_t\|x-y\|,$$ where
$$C_t:=\sqrt{2\sup_{k\ge1}\bigg[\bigg(\int_0^\infty e^{-\frac{\cos1}{2}tf_k(r)}\,dr\bigg)\wedge\bigg(e^{-\gamma_kt}\int_0^\infty
e^{-\frac{(\cos1)(1-e^{-2\gamma_kt})}{4\gamma_k}f_k(r)}\,dr\bigg)\bigg]},\quad t>0.$$
 \end{theorem}

The coupling property and gradient estimates have been intensively
studied for linear stochastic differential equations driven by
L\'{e}vy processes on $\R^d$, see e.g.\ \cite{ BSW,SSW, SW1,SW2,W1,W2}. Among them, gradient estimates for linear evolution equations on $\R^d$ have
been obtained in \cite[Theorem 1.1]{W2} by using a similar bound
condition \eqref{th22} on the L\'{e}vy measure, and then they are
improved in \cite[Theorem 1.3]{SSW} and \cite[Theorem 3.1]{SW2} via
the symbol of L\'{e}vy processes. Recently, by using the lower bound conditions for the L\'{e}vy
measure with respect to a nice reference probability measure, we
have successfully obtained the coupling property and gradient
estimates for linear stochastic differential equations driven by
non-cylindrical L\'{e}vy processes on Banach spaces, see
\cite[Theorem 1.2]{WW}. However, as mentioned in the paragraph
before \cite[Theorem 1.1]{WW}, the situation for cylindrical
L\'{e}vy processes is essentially different from that for
non-cylindrical L\'{e}vy processes. This implies that the framework
adopted in \cite{WW} does not apply to the present settings. We stress that Theorem \ref{th2} improves
\cite[Theorem 4.14]{PZ1} for cylindrical symmetric $\alpha$-stable noise, and
also points out the difference form non-cylindrical L\'{e}vy
processes.

As an application of Theorem \ref{th2}, we will study the exponential ergodicity for the stochastic process
$X=(X_t^x)_{t\ge 0}$ with
components \eqref{th2000}. Let $(f_n)_{n\ge1}$ be a sequence of Bernstein functions, and for any $t>0$, let $C_t$ be the constant defined in Theorem \ref{th2}.

\begin{theorem}\label{ergociff}Suppose that for any $n\ge1$, $(Z_t^n)_{t\ge0}$ in \eqref{th2000} is a symmetric pure jump L\'{e}vy process associated with the (symmetric) L\'{e}vy measure $\nu_n$. Assume that the following two assumptions hold:
\begin{itemize}
\item[(i)] There exists a constant $\alpha\in (0,1]$ such that$$\sum_{n=1}^\infty \int_0^\infty\bigg(e^{-2\gamma_n s}\int_{\{|z|\le e^{\gamma_ns}\}}z^2\,\nu_n(dz)+ e^{-\alpha\gamma_ns}\int_{\{|z|>e^{\gamma_ns}\}}|z|^\alpha\,\nu_n(dz)\bigg)\,ds<\infty;$$

\item[(ii)] For any $n\ge1$, $$\nu_n(dz)\ge |z|^{-1}f_n(z^{-2})\,dz,$$ where
$f_n$ is a Bernstein function such that $f_n(0)=0$, $\lim\limits_{r\to
\infty}\frac{f_n(r)}{\log (1+r)}=\infty$ and $\lim\limits_{t\to\infty}C_t=0$.
\end{itemize} Then, the process $X=(X_t^x)_{t\ge 0}$ with
components \eqref{th2000} takes values in $\Hh$, and there exist
the unique invariant measure $\mu$ and a constant $C>0$ such that for any $t>0$ and
$x\in\Hh$,
$$\|P_t(x,\cdot)-\mu\|_{\var}\le  C(1+\|x\|^\alpha)C_t^\alpha.$$
 \end{theorem}

\bigskip

The remaining part of this paper is organized as follows. In Section
\ref{section2} we present some preliminaries on cylindrical L\'{e}vy
processes. Section
\ref{section4} is devoted to the proof of Theorem \ref{th2}. Here,
the most important estimates for the density function of
one-dimensional subordinate Brownian motions are established in
Proposition \ref{grth1}, which is key to the proof of Theorem
\ref{th2} and indicates the difference from finite dimensional
situations. We also obtain the gradient estimate and the coupling
property for the linear evolution equation driven by cylindrical
subordinate Brownian motions, see Proposition \ref{thsub} below. In
Section \ref{section5}, we will study the exponential ergodicity for
the linear evolution equations \eqref{eq100}. The proof of Theorem
\ref{ergociff} is presented.
\section{Preliminarily: Cylindrical L\'{e}vy Processes on Hilbert Spaces}\label{section2}
In the section, $\Hh$ will denote a real separable Hilbert space
with the inner product $\langle \cdot,\cdot\rangle$ and the norm
$\|\cdot\|$. We will fix an orthonormal basis $(e_n)$ in $\Hh$.
Through the basis $(e_n)$ we will often identify $\Hh$ with $l^2$.
More generally, for a given sequence $\rho=(\rho_n)$ of real
numbers, we set
\begin{equation}\label{p1}l_\rho^2:=\bigg\{(x_n)\in\R^\N: \sum_{n=1}^\infty
x_n^2\rho_n^2<\infty\bigg\}.\end{equation} It is clear that
$l_\rho^2$ becomes a separable Hilbert space with the inner product
$\langle x,y\rangle=\sum\limits_{n=1}^\infty x_ny_n\rho_n^2$ for
$x=(x_n), y=(y_n)\in l_\rho^2.$

Let us recall that a L\'{e}vy
process $Z=(Z_t)_{t\ge0}$ with values in $\Hh$ is an $\Hh$-valued
process defined on some stochastic basis $(\Omega,\mathscr{F},
(\mathscr{F}_t), \Pp)$, continuous in probability, having stationary
independent increments, c\`{a}dl\`{a}g trajectories, and such that
$Z_0=0$, $\Pp$-a.s. It is well known that
$$\Ee e^{i\langle Z_t, x\rangle} =\exp \big(-t\Phi(x)\big),\quad x\in \Hh,$$
where the characteristic exponent (or the symbol) $\Phi$ can be
expressed by the following infinite dimensional L\'{e}vy-Khintchine
formula
\begin{equation*}\label{p2} \Phi(x)=\frac{1}{2}\langle Q x,x\rangle-i\langle b,x\rangle+
\int_{\Hh}\bigg(1-e^{i\langle x,z\rangle}+i\langle
x,z\rangle\I_{\{\|z\|\le 1\}}\bigg)\,\Pi(dz),\quad x\in
\Hh.\end{equation*} Here $Q$ is a non-negative, self-adjoint trace
class operator on $\Hh$, $b\in \Hh$ and $\Pi$ is the L\'{e}vy
measure on $\Hh\backslash \{0\}$ such that $\int_{\Hh\backslash
\{0\}} (1\wedge \|z\|^2)\,\Pi(dz)<\infty.$ We call the triple
$(Q,b,\Pi)$ the characteristics of the L\'{e}vy process $Z$, e.g.\
see \cite[Chapter VI]{Pa} and \cite[Chapter 4]{PZ}.

For cylindrical L\'{e}vy process $Z$ given
by \eqref{eq2}, we assume that $(Z_t^n)_{t\ge0, n\ge1} $ are defined
on the same stochastic basis $(\Omega, \mathscr{F}, (\mathscr{F}_t),
\Pp)$ satisfying the usual assumptions. Since for any $n\ge1$,
$(Z_t^n)_{t\ge0}$ is a pure jump L\'{e}vy process on $\R$, we have,
for each $n\ge1$ and $t\ge0$,
$$\Ee e^{ih Z_t^n}= e^{-t\psi_n(h)},\,\quad h\in \R,$$ where
\begin{equation}\label{levy1}\psi_n(h)=
\int_{\R} \Big( 1-e^{ihz}+ihz\I_{\{|z|\le 1\}}\Big)\, \nu_n(dz),\quad h\in \R, n\ge1,\end{equation}
 and $\nu_n$ is the L\'{e}vy measure on $\R\backslash \{0\}$ with
 $\int_{\R\backslash \{0\}}(1\wedge z^2)\,\nu_n(dz)<\infty,$ e.g.\ see \cite{Sa}.
The following result essentially follows from \cite[Theorem
4.40]{PZ}, which gives us an if and only if condition such that
the cylindrical L\'{e}vy process $Z=(Z_t)_{t\ge0}$ given by
\eqref{eq2} takes values in $\Hh$.
\begin{proposition}\label{thlevy}{$($\cite[Theorem
4.40]{PZ}$)$}\,\,The cylindrical L\'{e}vy process $Z$ given by
\eqref{eq2} takes values in $\Hh$ if and only if
\begin{equation}\label{not0}\sum_{n=1}^\infty \int(1\wedge
z^2)\,\nu_n(dz)<\infty.\end{equation}
\end{proposition}

As a direct consequence, we have the following example.

\begin{example}\label{thlevy1} Let $Z=(Z_t)_{t\ge0}$ be a cylindrical L\'{e}vy process with the following form
\begin{equation} \label{eq22} Z_t=\sum_{n=1}^\infty \beta_nZ_t^ne_n,\quad t\ge0,\end{equation}
where $(e_n)$ is an orthonormal basis of $\Hh$,
$(Z_t^n)_{t\ge0,n\ge1}$ are independent, real valued, identically
distributed L\'{e}vy processes defined on a fixed stochastic basis,
and $(\beta_n)$ is a given possibly unbounded sequence of positive
numbers. Let $\nu$ be the common L\'{e}vy measure corresponding to
$(Z_t^n)_{t\ge0,n\ge1}$, e.g.\ see \eqref{levy1}. Then, the
cylindrical L\'{e}vy process $Z$ given by \eqref{eq22} takes values
in $\Hh$ if and only if
\begin{equation*}\label{not1}\sum_{n=1}^\infty \bigg(\beta_n^2\int_{\{|z|\le 1/\beta_n\}}z^2\,\nu(dz)+\int_{\{|z|> 1/\beta_n\}}\nu(dz)\bigg)<\infty.\end{equation*}
 \end{example}

Example \ref{thlevy1} covers \cite[Proposition 3.3]{PZ1} for
cylindrical symmetric $\alpha$-stable processes, and also extends
\cite[Proposition 2.4]{PZ2} where for any $n\ge1,$ the symmetry
of the L\'{e}vy process $(Z^n_t)_{t\ge0}$ is required. On the other
hand, according to Proposition \ref{thlevy}, we know that the
cylindrical L\'{e}vy process $Z$ given by \eqref{eq22} is a L\'{e}vy
process with values in the space $l_\rho^2$, see \eqref{p1}, where
$(\rho_n)$ is a sequence of positive numbers such that
$$\sum_{n=1}^\infty \bigg((\rho_n\beta_n)^2\int_{\{|z|\le 1/(\rho_n\beta_n)\}}z^2\,\nu(dz)+\int_{\{|z|> 1/(\rho_n\beta_n)\}}\nu(dz)\bigg)<\infty.$$ That is, with the sequence $(\rho_n)$ above, the cylindrical L\'{e}vy
process $Z$ given by \eqref{eq22} is an honest $l^2_{\rho}$-valued L\'{e}vy process, see \cite[Remark 4.1]{PZ1} and \cite[Remark 2.7]{PZ2}.

Furthermore, as an application of Proposition \ref{thlevy}, we can
provide a criteria such that the process $X=(X_t^x)_{t\ge0}$ with
components \eqref{th2000} takes values in $\Hh$.

\begin{corollary}\label{thex1} The process $X=(X_t^x)_{t\ge0}$ with components \eqref{th2000} takes values in $\Hh$ if and only if \eqref{not2} holds.  Moreover, under
\eqref{not2}, $X=(X_t^x)_{t\ge0}$ is a Markov process, and
$$X_t^x=\sum_{n=1}^\infty X_t^{x,n}e_n=e^{tA}x+Z_A(t),$$ where
$$Z_A(t)=\int_0^te^{(t-s)A}\,dZ_s=\sum_{n=1}^\infty\left(\int_0^te^{-\gamma_n(t-s)}\,dZ_s^n\right)e_n.$$ \end{corollary}

When the L\'{e}vy process $(Z_t^n)_{t\ge0}$ is square integrable
with zero mean, Corollary \ref{thex1} has been proved in
\cite[Proposition 9.7]{PZ}; when the L\'{e}vy process
$(Z_t^n)_{t\ge0}$ is symmetric,  Corollary \ref{thex1} is just
\cite[Theorem 2.8]{PZ2}. For the general case, we refer to
\cite[Corollary 6.3]{Ri} for the recent study. Below we include the
proof of Corollary \ref{thex1} for the sake of completeness.

\begin{proof}[Proof of Corollary $\ref{thex1}$]
For any $n\ge1$ and $t\ge0$, let us consider the stochastic convolution
$$Y_t^n=\int_0^te^{-\gamma_n(t-s)}\,dZ_s^n.$$ The law of $Y_t^n$ is an infinitely divisible probability distribution, and its characteristic exponent is
   $$-\log(\Ee(e^{ihY_t^n}))=\int_0^t\psi_n(e^{-\gamma_ns}h)\,ds,\quad t\ge0,h\in\R,$$ where $\psi_n$ is given by
\eqref{levy1}, e.g.\ see \cite[Corollary 4.29]{PZ} or \cite[Lemma 3.1]{SW2}.
    Since the driving L\'evy process $(Z^n_t)_{t\ge0}$ is a pure jump L\'{e}vy process, the L\'evy {triplet} $(0,b_t,\nu_{n,t})$ of $Y_t^n$ is given by, e.g.\ see \cite[Theorem 3.1]{SAT} and \cite[Proposition 2.1]{Mu},
    \begin{gather*}
        \nu_{n,t}(D)= \int_0^t\nu_n(e^{-s\gamma_n}D)\,ds,\qquad D\in\mathscr{B}(\R\setminus\{0\}),\\
        b_t= \int_{z\neq 0} \int_0^t e^{s\gamma_n}z\Big(\I_{\{|z|\le 1\}} \big(e^{s\gamma_n}z\big) - \I_{\{|z|\le1\}}(z) \Big)\,ds\,\nu_n(dz),
    \end{gather*} where $\nu_n$ is the L\'evy measure of $\psi_n$ given in \eqref{levy1}.
Therefore,
\begin{equation}\label{profthex11}\aligned &\sum_{n=1}^\infty \int (1\wedge z^2)\,\nu_{n,t}(dz)\\
&=\sum_{n=1}^\infty \int_0^t\left(\int_\R(1\wedge(e^{-s\gamma_n}z)^2)\,\nu_n(dz)\right)\,ds\\
&=\sum_{n=1}^\infty \int_0^t\bigg(e^{-2\gamma_n s}\int_{\{|z|\le e^{\gamma_ns}\}}z^2\,\nu_n(dz)+ \int_{\{|z|>e^{\gamma_ns}\}}\,\nu_n(dz)\bigg)\,ds.\endaligned\end{equation}

Assumption \eqref{not2} gives us that $\Pp(X_1^x\in \Hh)=1$, due to \eqref{profthex11} and \eqref{not0}.
So, in order to complete the proof, it remains to show that $\Pp(X_1^x\in
\Hh)=1$ implies $\Pp(X_t^x\in \Hh)=1$ for any $t>0$. Since \eqref{not2} verifies that
$$\sum_{n=1}^\infty \int_0^t\bigg(e^{-2\gamma_n s}\int_{\{|z|\le e^{\gamma_ns}\}}z^2\,\nu_n(dz)+ \int_{\{|z|>e^{\gamma_ns}\}}\,\nu_n(dz)\bigg)\,ds<\infty,\quad 0<t\le1,$$ also by \eqref{not0}, for any $0<t\le 1$, $\Pp(X_t^x\in \Hh)=1$.

For any $t>0$, set
$Y_t=\big(Y_t^n\big)_{n\ge1}$. Then, $$Y_t=\sum_{n=0}^\infty\left(\int_0^te^{-\gamma_n(t-u)}\,dZ_u^n\right)e_n=:\int_0^te^{(t-u)A}\,dZ_u.$$ For any $t>1$, we have the following identity on the product space
$\R^{\N}$:
$$\aligned
Y_t-e^{(t-1)A}Y_1&=\int_0^te^{(t-u)A}\,dZ_u-e^{(t-1)A}\int_0^1e^{(1-u)A}\,dZ_u\\
&=\int_0^te^{(t-u)A}\,dZ_u-\int_0^1e^{(t-u)A}\,dZ_u\\
&=\int_1^te^{(t-u)A}\,dZ_u\\
&=\int_0^{t-1}e^{(t-1-u)A}\,dZ_u^1,\endaligned$$ where
$Z_u^1=Z_{1+u}-Z_1$, $u\ge0$, is still a L\'{e}vy process with values in
$\R^{\N}$. Note that $\int_0^{t-1}e^{(t-1-u)A}\, dZ_u^1$ has the
same law as $Y_{t-1}$. Since $\Pp(Y_{t-1}\in \Hh)=1$ for any $1<t\le 2$
and $\Pp(Y_1\in\Hh)=1$, we have, by the identity above, $\Pp(Y_t\in \Hh)=1$ for any $1<t\le 2.$ Using an
iteration produce, we furthermore infer that $\Pp(Y_t\in \Hh)=1$ for any
$t\ge0$. This proves the required assertion.

The Markov property of $X$ follows from the identity below
$$Y_t-e^{(t-s)A}Y_s= \int_0^{t-s}e^{(t-s-u)A}\,dZ_u^s,\quad 0<s\le t,$$ where $Z_u^s=Z_{s+u}-Z_s$.
 \end{proof}

We close this section with two remarks and one example for Corollary \ref{thex1}.

\begin{remark} (1) According to \cite[Theorem 2.3]{Ch} or
\cite[Theorem 7]{App}, if the L\'{e}vy process $Z$ takes values in
$\Hh$, then the linear equation \eqref{eq100} has the unique
mild solution $X$ which also takes values in $\Hh$. Thus,
it is not surprising to see that \eqref{not0} implies
\eqref{not2}.

(2) If the
cylindrical L\'{e}vy process $Z$ given by \eqref{eq2} takes values
in the Hilbert space $\Hh$, i.e.\ the condition \eqref{not0} holds,
then, by the Kotelenez regularity result (see \cite[Theorem
9.20]{PZ}), trajectories of the process $X$ which solves
\eqref{eq100} are c\`{a}dl\`{a}g with values in $\Hh$, e.g.\ see
\cite[the first paragraph in Section 4.1]{PZ1} and \cite[Remark
2.10]{PZ2}. On the other hand, assume that the process $X$ which
solves \eqref{eq100} has $\Hh$-c\`{a}dl\`{a}g modification, where
the L\'{e}vy process $Z$ is defined by \eqref{eq2}. Then, according
to (the proof of) \cite[Theorem 2.1]{LZ}, for any $\varepsilon>0$,
$$\sum_{n=1}^\infty \int_{\{|z|\ge \varepsilon\}}\,\nu_n(dz)<\infty,$$ where for any $n\ge1$, $\nu_n$ is the L\'{e}vy measure corresponding to the pure jump L\'{e}vy process $(Z_t^n)_{t\ge0}$.
\end{remark}

\begin{example}[Continuation of Example \ref{thlevy1}] Let $X$ be the process with components \eqref{th2000}, where $Z$ is the cylindrical L\'{e}vy process given by \eqref{eq22}. Then, the process $X$ takes values in $\Hh$ if and only if
\begin{equation}\label{not3}\sum_{n=1}^\infty \int_0^1\bigg(\beta_n^2e^{-2\gamma_n s}\int_{\{|z|\le \beta_n^{-1}e^{\gamma_ns}\}}z^2\,\nu(dz)+ \int_{\{|z|>\beta_n^{-1}e^{\gamma_ns}\}}\,\nu(dz)\bigg)\,ds<\infty.\end{equation}  \end{example}

\section{Coupling Property and Gradient Estimates}\label{section4}
\subsection{One-dimensional Ornstein-Uhlenbeck Processes Driven by Subordinate Brownian Motions}
In this subsection, we will consider one-dimensional
Ornstein-Uhlenbeck processes driven by subordinate Brownian
motions, which are a class of special but important L\'{e}vy
processes.

We first recall some facts and properties for subordinate Brownian motions. Suppose that $(B_t)_{t\ge0}$ is a Brownian motion on $\R$ with
$$\Ee\Big[e^{ih(B_t-B_0)}\Big]=e^{-th^2},\qquad h\in\R,
t>0,$$ and $(S_t)_{t\ge0}$ is a subordinator (that is,
$(S_t)_{t\ge0}$ is a nonnegative L\'{e}vy process on $[0,\infty)$
such that $S_t$ is increasing and right-continuous in $t$ with
$S_0=0$) independent of $(B_t)_{t\ge0}$. For any $t\ge0$, let
$\mu_t^S$ be the probability distribution of the subordinator $S$,
i.e.\ $\mu_t^S(D)=\Pp(S_t\in D)$ for any $D\in\Bb([0,\infty))$. It
is well known that the associated Laplace transformation of
$\mu_t^S$ for each $t\ge0$ is given by
$$\int_0^\infty e^{-r s}\mu_t^S(ds)=e^{-tf(r)},\qquad r>0,$$
where $f(r)$ is a Bernstein function, i.e.\ $(-1)^{k} f^{(k)}(r)\le 0$ for all $r>0$ and $k\ge 1$. We refer to \cite{SSZ}
for more details about Bernstein functions and subordinators. The
process $Z=(Z_t)_{t\ge0}$ on $\R$ defined by $$Z_t=B_{S_t},\quad
t\ge0$$is called subordinate Brownian motion, which is a symmetric
L\'{e}vy process with
$$\Ee\Big[e^{ih(Z_t-Z_0)}\Big]=e^{-tf(h^2)},\qquad h\in\R,
t>0.$$ That is, the symbol or the characteristic exponent of the
subordinate Brownian motion $(Z_t)_{t\ge0}$ is $f(h^2)$, see
\cite{JBOOK}. The transition density function of the subordinated
Brownian motion $Z$ exists, and it is given by
\begin{equation}\label{density}p^Z_t(x,y)=p^Z_t(x-y)=\int_0^\infty \frac{1}{\sqrt{4\pi
s}}\exp\bigg(-\frac{|x-y|^2}{4s}\bigg)\,\mu^S_t(ds)\end{equation} for $t>0$ and $x,y\in\R.$
 Examples of subordinate Brownian motions include
symmetric $\alpha$-stable processes, relativistic
$\alpha$-stable processes and so on.

\medskip

Now, we will consider the process $X=(X_t)_{t\ge0}$ given by
\begin{equation}\label{gr1} X_t^x=e^{-\gamma t}x+\int_0^t e^{-\gamma(t-s)}\,dZ_s,\quad t\ge0, x\in\R,\end{equation}
where $\gamma\ge 0$ and $(Z_t)_{t\ge0}$ is a subordinated Brownian
motion on $\R$ associated with the Bernstein function $f$.

The following result is the key to the proof of Theorem \ref{th2}.
Denote by $C_b^\infty(\R)$ the set of all infinite differentiable
functions on $\R$ such that all their derivatives are bounded.
\begin{proposition}\label{grth1} If \begin{equation}\label{grth11}\lim_{r\to\infty}\frac{
f(r)}{\log(1+r)}=\infty,\end{equation} then the Ornstein-Uhlenbeck
process $(X^x_t)_{t\ge0}$ given by \eqref{gr1} has a density
function of the form $p_t(\cdot -e^{-\gamma_nt}x)$. Here, for any
$t>0$, the function $p_t$ enjoys the following properties:
\begin{itemize}
\item[(i)] $p_t>0$ and $p_t\in C_b^\infty(\R)$ such that $p_t$ is even
on $\R$, $p'_t$ is odd on $\R$.

\item[(ii)]\begin{equation}\label{grth2}\aligned \int_{\R} \frac{(p'_t(x))^2}{p_t(x)}\,dx &\le 2\int_0^\infty
e^{-F_t(r)}\,dr,
      \endaligned\end{equation}where for any $t,r>0$,
$$F_t(r)=\int_0^t f(e^{-2\gamma s }r)\,ds.$$
\end{itemize}
\end{proposition}
\begin{proof}
Let us first consider the stochastic convolution
$$Y_t=\int_0^t e^{-\gamma(t-s)}\,dZ_s,\quad t\ge0.$$ A direct calculation shows for any $h\in\R$ and $t\ge0$,
$$\Ee e^{-ihY_t}=\exp\bigg(-\int_0^t f(e^{-2\gamma s}h^2)\,ds\bigg)=\exp\bigg(-F_t(h^2)\bigg),$$
e.g.\ see \cite[Corollary 4.29]{PZ} or \cite[Lemma 3.1]{SW2}.

Since $r\mapsto f(e^{-2\gamma s }r)$ is a Bernstein function for every $s>0$ and
$\gamma\ge0$, it follows from
\cite[Corollary 3.7]{SSZ} that $F_t$ is also a Bernstein function.
Then, there exists a probability measure $\mu_t$ on $[0,\infty)$
such that its Laplace transformation
$$\int_0^\infty e^{-r s}\,\mu_t(ds)=e^{-F_t(r)},\,\quad r\ge0.$$

Furthermore, due to \eqref{grth11}, we have
$$\lim_{r\to\infty}\frac{F_t(r)}{\log(1+r)}=\infty\quad \textrm{ for all }
t>0.$$ According to (the proof of) \cite[Theorem 1]{SV}, we know
that for all $t>0$ there exists a density function $p_t$ for $Y_t$
such that $p_t\in C_b^\infty(\R)$, and $p_t^{(n)}$ exists and
belongs to $ L_1(\R)\cap C_\infty(\R)$ for all $n\ge0.$ By
\eqref{gr1}, we know that for any $t>0$, the density function of
$X_t^x$ exists and is equivalent to $p_t(\cdot -e^{-\gamma t}x).$ On
the other hand, as mentioned above, $F_t(r)$ is a Bernstein
function. Therefore, according to \eqref{density},
$$p_t(x)=\int_0^\infty \frac{1}{\sqrt{4\pi s}}e^{-\frac{x^2}{4s}}\,\mu_t(ds),\quad x\in\R.$$
Clearly, $p_t(x)>0$ for all $x\in\R$, $p_t(x)$ is even and $p'_t(x)$
is odd. In what follows, we write $p_t(x)=q_t(x^2)$ for all
$x\in\R$, where
$$q_t(x) =\int_0^\infty \frac{1}{\sqrt{4\pi
s}}e^{-\frac{x}{4s}}\,\mu_t(ds).$$ It is easy to see that for any
$n\ge1$, $q_t^{(n)}$ exists such that for any $x>0$,
$q_t^{(2n-1)}(x)<0$ and $q_t^{(2n)}(x)>0$. Therefore,
$$\aligned \int_{\R} \frac{(p'_t(x))^2}{p_t(x)}\,dx &=2\int_0^\infty \frac{(p'_t(x))^2}{p_t(x)}\,dx=4\int_0^\infty \frac{\sqrt{x}q_t'(x)^2}{q_t(x)}\,dx .\endaligned$$
Since, by $q(+\infty)=q'(+\infty)=0$, $q'_t(x)<0$ and
$q^{(3)}_t(x)<0$ for all $x>0$,
$$q'_t(x)^2=-2\int_{x}^\infty q'_t(u)q''_t(u)\,du\le
2q''_t(x)\bigg(-\int_x^\infty q'_t(u)\,du\bigg)=2 q''_t(x)q_t(x),$$
we arrive at
$$\int_{\R} \frac{(p'_t(x))^2}{p_t(x)}\,dx \le 8\int_0^\infty \sqrt{x} q''_t(x)\,dx.$$
Noting that
$$q''_t(x)=\frac{1}{32\sqrt{\pi}}\int_0^\infty \frac{1}{s^{5/2}}e^{-\frac{x}{4s}}\,\mu_t(ds),$$ we further get
$$\aligned\int_{\R} \frac{(p'_t(x))^2}{p_t(x)}\,dx\le& \frac{1}{4\sqrt{\pi}}\int_0^\infty \sqrt{x}\int_0^\infty
\frac{1}{s^{5/2}}e^{-\frac{x}{4s}}\,\mu_t(ds)\,dx\\
=& \frac{2}{\sqrt{\pi}}\int_0^\infty
\sqrt{r}e^{-r}\,dr\int_0^\infty \frac{1}{s}\mu_t(ds)\\
=&2\int_0^\infty \frac{1}{s}\mu_t(ds)\\
=&2\int_0^\infty\int_0^\infty e^{-r s}\, dr\,\mu_t(ds)\\
=&2\int_0^\infty \,dr\int_0^\infty e^{-rs}\,\mu_t(ds)\\
=&2\int_0^\infty e^{-F_t(r)}\,dr.\endaligned$$ This proves
\eqref{grth2}. The proof is completed.
\end{proof}

\subsection{Ornstein-Uhlenbeck Processes on $\Hh$ Driven by Cylindrical
Subordinated Brownian Motions} In this part, let $X$ be the process with components \eqref{th2000}, where for any $n\ge1$, $(Z_t^n)_{t\ge0}$ is a subordinated Brownian motion
associated with the Bernstein function $f_n$ satisfying
$$\lim_{r\to \infty}\frac{ f_n(r)}{\log (1+r)}=\infty.$$ For simplicity, we call $X$
the Ornstein-Uhlenbeck process driven by
additive cylindrical subordinate Brownian motions.

We will drive explicit estimates about the uniform norm of the
gradient for the semigroup $(P_t)_{\ge0}$ corresponding to linear
evolution equations driven by additive cylindrical subordinate
Brownian motions. Recall that the uniform norm of the gradient for
the semigroup $(P_t)_{t\ge0}$ is defined by
$$\|\nabla P_t\|_\infty:=\sup\bigg\{ |\nabla_zP_tg(x)|: \|z\|\le 1, x\in \Hh, g\in B_b(\Hh)\textrm{ with } \|g\|_{\Hh,\infty}\le1\bigg\},$$ where
$$|\nabla_zP_tg(x)|:=\limsup_{\varepsilon\to 0}\frac{1}{\varepsilon}|P_tg(x+\varepsilon z)-P_tg(x)|.$$

\begin{proposition}\label{thsub} Let $X$ be the Ornstein-Uhlenbeck process driven by
additive cylindrical subordinate Brownian motions as above, such that it takes values in $\Hh$. Let $(P_t)_{\ge0}$ and
$P_t(x,\cdot)$ be its semigroup and transition kernel, respectively. Then, for any
$t>0,$ and $x,y\in \Hh$,
\begin{equation}\label{thsub12211} \|\nabla P_t\|_\infty\le A_t,\end{equation} and
\begin{equation}\label{thsub12222} \|P_t(x,\cdot)-P_t(y,\cdot)\|_{\var}\le 2A_t\|x-y\|,\end{equation}where $$A_t= \sqrt{2\sup_{k\ge1}\bigg(e^{-2\gamma_kt}\int_0^\infty
e^{-F_{k,t}(r)}\,dr\bigg)}.$$ Moreover, it holds that for any $t>0$,
$$A_t\le C_t:=\sqrt{2\sup_{k\ge1}\bigg[\bigg(\int_0^\infty e^{-tf_k(r)}\,dr\bigg)\wedge\bigg(e^{-\gamma_kt}\int_0^\infty
e^{-\frac{1-e^{-2\gamma_kt}}{2\gamma_k}f_k(r)}\,dr\bigg)\bigg]},\quad t>0.$$
\end{proposition}

\begin{proof} Without loss of generality, we will assume that $A_t<\infty$ for all $t>0$.  The proof of \eqref{thsub12211} is motivated from that of \cite[Theorem 4.14]{PZ1}, and so here we only point out the main differences.
Identifying $\Hh$
with $l^2$ through the basis $(e_n)$, we have
$$g(x)={g}(x_1, \cdots, x_n,\cdots),\quad x\in \Hh.$$
 For any $x=(x_n)\in l^2$, denote by
$$\mu_t^x=\prod_{k=1}^\infty\mu_{k,t}^{x_k}$$ the Borel product
measure in $\R^\N$ such that
$$\mu_{k,t}^{x_k}(dz_k)=p_{k,t}(z_k-e^{-\gamma_ktx_k})\,dz_k,\quad k\ge1.$$
According to Proposition \ref{grth1} and Hypothesis {\bf(H)}, we get that for any $g\in B_b(\Hh)$,
$$P_tg(x)=\int_{\Hh}g(y)\prod_{k=1}^\infty\mu_{k,t}^{x_k}(dy_k)=\int_{\R^\N}{g}(z)\prod_{k=1}^\infty p_{k,t}(z_k-e^{-\gamma_kt}x_k)\,dz_k.$$

We first assume that $g\in C_b(\Hh)$, where $C_b(\Hh)$ denotes the class of all bounded continuous functions on $\Hh.$
For any $x, h\in\Hh$, denote by $D_hP_tg(x)$ the directional
derivative of $P_tg$ at $x$ along the direction $h$, i.e.
$$D_hP_tg(x)=\lim_{\varepsilon\to0}\frac{P_tg(x+\varepsilon h)-P_tg(x)}{\varepsilon}.$$
We will prove that $D_hP_tg(x)$ exists for any $x,h\in\Hh$. For
$n\ge1$, set $h_{(n)}=\sum_{k=1}^n h_ke_k$, so that $h_{(n)}\to h$
in $\Hh$ as $n\to\infty.$ By the dominated convergence theorem, it
follows that
$$\aligned D_{h_{(n)}} P_tg(x)=\!&\lim_{\varepsilon\to0}\frac{P_tg(x+\varepsilon h_{(n)})-P_tg(x)}{\varepsilon}\\
=\!&-\!\int_{\R^\N}{g}(z)
\left(\sum_{k=1}^n\frac{p'_{k,t}(z_k-e^{-\gamma_kt}h_k)}{p_{k,t}(z_k-e^{-\gamma_kt}h_k)}e^{-\gamma_kt}h_k\right)
\prod_{k=1}^\infty p_{k,t}(z_k-e^{-\gamma_kt}x_k)\,dz_k\\
=\!&-\!\int_{\Hh}g(z)\left(\sum_{k=1}^n\frac{p'_{k,t}(z_k-e^{-\gamma_kt}h_k)}{p_{k,t}(z_k-e^{-\gamma_kt}h_k)}
e^{-\gamma_kt}h_k\right)\prod_{k=1}^\infty\mu_{k,t}^{x_k}(dz_k).\endaligned$$
In order to pass to limit, as $n\to\infty$, we will show that
$$\sum_{k=1}^n\frac{p'_{k,t}(z_k-e^{-\gamma_kt}h_k)}{p_{k,t}(z_k-e^{-\gamma_kt}h_k)}
e^{-\gamma_kt}h_k\quad\textrm{ converges in }L^2(\mu_t^x).$$

Indeed, Proposition \ref{grth1} shows that the function $p_{k,t}$ is even, and its derivative $p'_{k,t}$ is
odd. Hence, for any $j\neq k,$
$$\aligned &e^{-\gamma_kt}e^{-\gamma_jt}\int_{\R^2}\frac{p'_{k,t}(z_k-e^{-\gamma_kt}h_k)}
{p_{k,t}(z_k-e^{-\gamma_kt}h_k)}\frac{p'_{k,t}(z_j-e^{-\gamma_jt}h_j)}
{p_{k,t}(z_j-e^{-\gamma_jt}h_j)}\\
&\qquad\qquad \qquad\times
p_{k,t}(z_k-e^{-\gamma_kt}h_k)p_{k,t}(z_j-e^{-\gamma_jt}h_j)\,dz_k\,dz_j=0.\endaligned$$
For any $n\ge1$ and $t,r>0$, define $$F_{n,t}(r)=\int_0^t
f_n(e^{-2\gamma_n s }r)\,ds.$$ Therefore, for any $m,j\in\N$,
$$\aligned &\int_{\Hh}\left(\sum_{k=m}^{m+j} \frac{p'_{k,t}(z_k-e^{-\gamma_kt}h_k)}{p_{k,t}(z_k-e^{-\gamma_kt}h_k)}e^{-\gamma_kt}h_k\right)^2\mu_t^x(dz)\\
&=\int_{\R^{j+1}}\left(\sum_{k=m}^{m+j}
\frac{p'_{k,t}(z_k-e^{-\gamma_kt}h_k)}{p_{k,t}(z_k-e^{-\gamma_kt}h_k)}e^{-\gamma_kt}h_k\right)^2\,
\prod_{k=m}^{m+j}p_{k,t}(z_k-e^{-\gamma_kt}x_k)\,dz_k\\
&=\int_{\R^{j+1}}\left(\sum_{k= m}^{m+j}
\frac{(p'_{k,t}(z_k-e^{-\gamma_kt}h_k))^2}{p^2_{kt}(z_k-e^{-\gamma_kt}h_k)}e^{-2\gamma_kt}h^2_k\right)\,
\prod_{k=m}^{m+j}p_{k,t}(z_k-e^{-\gamma_kt}x_k)\,dz_k\\
&=\sum_{k= m}^{m+j}\int_{\R^{j+1}}
\frac{(p'_{k,t}(z_k-e^{-\gamma_kt}h_k))^2}{p^2_{kt}(z_k-e^{-\gamma_kt}h_k)}e^{-2\gamma_kt}h^2_k
\prod_{i=m}^{m+j}p_{i,t}(z_i-e^{-\gamma_it}x_i)\,dz_i\\
&=\sum_{k=m}^{m+j}e^{-2\gamma_kt}h_k^2 \int
\frac{(p'_{k,t}(x))^2}{p_{k,t}(x)}\,dx\\
&\le\bigg[ 2\sup_{k\ge1}\bigg(e^{-2\gamma_kt}\int_0^\infty
e^{-F_{k,t}(r)}\,dr\bigg)\bigg] \sum_{k=m}^{m+j}h_k^2\\
&=:A_t^2 \sum_{k=m}^{m+j}h_k^2 ,\endaligned$$ where in the inequality above we have used \eqref{grth2}.

Note that, by changing variables, we find that for any $n\in \N$,
$$D_{h_{(n)}} P_tg(x)=-\int_{\Hh}
g(z+T_tx)\left(\sum_{k=1}^n\frac{p'_{k,t}(z_k)}{p_{k,t}(z_k)}e^{-\gamma_kth_k}\right)\,\mu_t^0(dz).$$
Therefore, we get that
$$D_hP_tg(x):=\lim_{n\to \infty} D_{h_{(n)}}
P_tg(x)=-\int_{\Hh}g(z+T_tx)\left(\sum_{k=1}^\infty\frac{p'_{k,t}(z_k)}{p_{k,t}(z_k)}e^{-\gamma_kth_k}\right)\,\mu_t^0(dz),$$
and for any $0<s<1$,
$$|D_{h(n)}P_tg(x+sh_{(n)})|\le A_t\|g\|_{\Hh,\infty}\|h\|.$$ Since $$P_tg(x)=\int_{\R^{\N}} g(z+T_tx)\,\mu_t(dz),$$ $(P_t)_{t\ge0}$ is Feller, i.e.\ for any
$f\in C_b(\Hh)$, $P_tf\in C_b(\Hh)$ for any $t\ge0$. Hence, by the fact that
$${P_tg(x+h_{(n)})-P_tg(x)}=\int_0^1 D_{h_{(n)}}
P_tg(x+sh_{(n)})\,ds$$ and again by the dominated convergence
theorem, we can pass to the limit, as $n\to \infty$, and get that
$${P_tg(x+h)-P_tg(x)}=\int_0^1 D_{h}
P_tg(x+sh)\,ds$$ and for any $0<s<1$,
$$|D_{h}P_tg(x+sh)|\le A_t\|g\|_{\Hh,\infty}\|h\|.$$
Therefore, for any $g\in C_b(\Hh)$, it holds that
\begin{equation*}\label{proofthsub}|{P_tg(x+h)-P_tg(x)}|\le A_t\|g\|_{\Hh,\infty}\|h\|.\end{equation*}
According to \cite[Lemma 7.1.5]{PZ0}, the estimate above also holds
for any $g\in B_b(\Hh)$. Hence, we prove the required
assertion \eqref{thsub12211}.

Furthermore, according to the definition of $\|\nabla P_t\|_{\infty}$,
there is a constant $\varepsilon_0>0$ such that for any $x$, $z\in
\Hh$ with $\|z\|\le 1$, $t>0$, $0<\varepsilon\le \varepsilon_0$ and
$g\in B_b(\Hh)$,
$$|P_tg(x+\varepsilon z)-P_tg(x)|\le 2\varepsilon\|\nabla P_t\|_\infty \|z\|\|g\|_{\Hh, \infty}.$$ On the other hand, for any $x$, $y\in\Hh$ with $x\neq y$, we can choose $n$ large enough such that $\|x-y\|/n\le \varepsilon_0$ and set $x_i=x+i(y-x)/n$ for $i=0, \ldots, n$. Thus,
\begin{equation}\label{www}\aligned  \|P_t(x,\cdot)-P_t(y,\cdot)\|_{\var}\le &\sum_{i=1}^n\|P_t(x_i,\cdot)-P_t(x_{i-1},\cdot)\|_{\var}\\
=&\sum_{i=1}^n\sup_{\|g\|_{\Hh,\infty}\le 1}|P_tg(x_i)-P_tg(x_{i-1})|\\
\le&2\|\nabla P_t\|_\infty\sum_{i=1}^n\|x_i-x_{i-1}\|\\
=& 2\|\nabla P_t\|_\infty\|x-y\|.\endaligned\end{equation} Then, the
desired assertion \eqref{thsub12222} is a consequence of
\eqref{thsub12211} and \eqref{www}.

Finally, for any $k\ge1$ and $r, t>0$,
$$\aligned F_{k,t}(r)\ge& tf_k(e^{-2\gamma_k t}r)\endaligned$$ and $$\aligned F_{k,t}(r)\ge&\int_0^{t/2}f_k(e^{-2\gamma_k s}r)\,ds
=\int_{e^{-\gamma_k t}}^1\frac{1}{2\gamma_k u}f(ur)\,du
\ge\frac{1-e^{-\gamma_k t}}{2\gamma_k}f(e^{-\gamma_k
t}r).\endaligned$$ Thus, for any $t>0$,
\begin{equation}\label{psps1}A_t^2\le 2 \sup_{k\ge1}\bigg[\bigg(\int_0^\infty e^{-tf_k(r)}\,dr\bigg)\wedge\bigg(e^{-\gamma_kt}\int_0^\infty
e^{-\frac{1-e^{-2\gamma_kt}}{2\gamma_k}f_k(r)}\,dr\bigg)\bigg]=C_t^2<\infty.\end{equation} This proves the last assertion.
\end{proof}

At the end of this subsection, we take the following example by
applying Proposition \ref{thsub}.
\begin{example} Let $X$ be the process with components \eqref{th2000}, where the cylindrical L\'{e}vy process $Z$ is of the form \eqref{eq22}, and for any $n\ge1$, $(Z_t^n)_{t\ge0}$ is a subordinated Brownian motion
associated with the common Bernstein function $f$ satisfying
$$\lim_{r\to \infty}\frac{ f(r)}{\log (1+r)}=\infty,\quad \lim\limits_{s\to0\textrm{ or } s\to \infty}\frac{f(\lambda s)}{f(s)}\in(0,\infty)\quad\textrm{ for }\lambda>0.$$ Suppose that \begin{equation}\label{subnot1not}\sum_{n=1}^\infty \int_0^1\bigg(\beta_n^2e^{-2\gamma_n s}\int_0^{ \beta_n^{-2}e^{2\gamma_ns}}f(r^{-1})\,dr+ \int_{\beta_n^{-2}e^{2\gamma_ns}}^\infty r^{-1}f(r^{-1})\,dr\bigg)\,ds<\infty.\end{equation}
Then, the process $X$ is a Markov process taking values in $\Hh$, and for any
$t>0,$ $x,y\in \Hh$,
\begin{equation*}\label{thsub1} \|\nabla P_t\|_\infty\le C_t,\end{equation*} and
\begin{equation*}\label{thsub122} \|P_t(x,\cdot)-P_t(y,\cdot)\|_{\var}\le 2C_t\|x-y\|,\end{equation*} where $$C_t\!\!:=\!\!\sqrt{2}\Bigg\{\!\bigg(\frac{1}{\inf_{k\ge1}\beta_k}\sqrt{\int_0^\infty e^{-tf(r)}\,dr}\bigg)\wedge\sqrt{\sup_{k\ge1}\bigg(\frac{e^{-\gamma_kt}}{\beta_k^2}\int_0^\infty
e^{-\frac{1-e^{-2\gamma_kt}}{2\gamma_k}f(r)}\,dr\bigg)}\Bigg\},\,\, t>0.$$
 \end{example}
\begin{proof} We first give a upper bound for the L\'{e}vy measure of one-dimensional subordinate Brownian motion $B_{S_t}$, where
$S$ is a subordinator associated with the Bernstein function $f$ satisfying the assumptions in the example.
According to \cite[Lemma 1.1]{JS} (also see the proof of it in
\cite[Appendix 2]{JS}), the corresponding L\'{e}vy measure $\nu$
of $B_{S_t}$ is given by $\nu(dz)=\rho(z^2)\,dz$, where
\begin{equation}\label{couooo1}\aligned \rho(z^2)=&\frac{1}{(4\pi)^{1/2}}\int_0^\infty
s^{-1/2}\exp\Big(\!\!-\frac{z^2}{4s}\Big)\,\mu(ds),\endaligned\end{equation}
and $\mu$ is the L\'{e}vy measure corresponding to $f$, i.e.\
$$f(r)=\int_0^\infty \left(1-e^{-r t}\right)\,\mu(dt),$$ see \cite{SSZ} and \cite[the remark below Theorem 1.1]{BSW}.
Note that, it is equivalent to saying that $$\aligned
f(r)=&\int_0^\infty\int_0^tr e^{-r
s}\,ds\,\mu(dt)=r\int_0^\infty e^{-r
s}\mu(s,\infty)\,ds.\endaligned$$ Since for any $\lambda>0$,
$\lim\limits_{s\rightarrow\infty}f(\lambda s)/f(s)\in(0,\infty)$, by \cite[Theorem 1.4.1, Page 17;
Proposition 1.5.1, Page 22]{BSW1},
there exist $\rho\ge0$ and a positive function $l$ on
$[0,\infty)$ slowly varying at $\infty$ (i.e. $l$ satisfies
$\lim_{x\rightarrow\infty}l(\lambda x)/l(x)=1$ for each $\lambda>0$)
such that $f(s)=s^\rho l(s)$. Following \cite[Section 1.4.2,
Definitation, Page 38]{BSW1}), we call $f$ regularly varying at
infinity with index $\rho$. According to the Abelian and Tauberian
theorems (see \cite[Theorem 1.7.1, Page 37]{BSW1}),
$h_\mu(s):=\mu(s,\infty)$ is regularly varying at $0$ (with index
$-\rho$) and for $r$ large enough,
$$f(r)\asymp \mu\big(r^{-1},\infty\big).$$
Applying the Abelian and Tauberian theorems and the monotone density theorem
(see \cite[Theorem 1.7.2, Page 39]{BSW1}) to \eqref{couooo1} further
yields for $|z|$ small enough
$$\rho(z^2)\asymp|z|^{-1}\mu\big({4z^2},\infty\big).$$ Here we also have used the fact that $\tilde{h}_\mu(s):=\mu\big(\frac{1}{4s},\infty\big)$ is
regularly varying at $\infty$. Combining all the conclusions above
with the fact that $f$ is regularly varying at $\infty$, we
can choose two constants $r_1, c_1>0$ such that for $0<|z|\le
r_1,$
$$  |z|^{-1}f(z^{-2})\ge c_1\rho (z^2).$$
By following the same arguments as above and using the condition that any $\lambda>0$,
$\lim\limits_{s\rightarrow0}f(\lambda s)/f(s)\in(0,\infty)$ , one can verify that there are two constants $r_2, c_2>0$ such that for $|z|\ge
r_2,$
$$  |z|^{-1}f(z^{-2})\ge c_2\rho (z^2).$$ Since the functions $r\mapsto\rho(r^2)$ and $r\mapsto r^{-1}f(r^{-2})$ are continuous and positive on bounded interval $[r_1,r_2]$, we can finally find a constant $c_0>0$ such that
for any $|z|>0$, \begin{equation}\label{proofex1}|z|^{-1}f(z^{-2})\ge c_0\rho (z^2).\end{equation}

 By changing the variables, we find that condition \eqref{subnot1not} implies
  \begin{equation}\label{subnot1subnot1not}\sum_{n=1}^\infty \int_0^1\bigg(\beta_n^2e^{-2\gamma_n s}\int_0^{ \beta_n^{-1}e^{\gamma_ns}}rf(r^{-2})\,dr+ \int_{\beta_n^{-1}e^{\gamma_ns}}^\infty r^{-1}f(r^{-2})\,dr\bigg)\,ds<\infty.\end{equation}
  Therefore, according to \eqref{subnot1subnot1not}, \eqref{proofex1} and \eqref{not3},
  the process $X$ takes values in $\Hh$ and it is Markovian, also thanks to Corollary \ref{thex1}.

 For any $n\ge1$, the process $(\beta_nZ_t^n)_{t\ge0}$ is a subordinate Brownian motion with the
 Bernstein function $r\mapsto f_n(r):=f(\beta_n^2r)$. Then, the required assertions for gradient estimate and the coupling property follow from Proposition \ref{thsub}.
\end{proof}

\subsection{Ornstein-Uhlenbeck Processes on $\Hh$ Driven by Cylindrical
L\'{e}vy Processes} In this part, we will present the proof of Theorem \ref{th2} by using the split technique, e.g.\ see \cite[Theorem 1.1]{BSW} and \cite[Theorem 3.2]{SSW}.
\begin{proof}[Proof of Theorem $\ref{th2}$] By the assumption \eqref{th21} on $\nu_n(dz)$, we can get that $$Z_t^n= Y_t^n+ \overline{Z_t^n}, \quad n\ge1, t>0,$$ where $(Y_t^n)_{t\ge0}$ is a L\'{e}vy process on $\R$ with L\'{e}vy measure $$\nu_{n,Y}(dz)=|z|^{-1}f_n(|z|^{-2})\,dz,$$ and
 $(\overline{Z_t^n})_{t\ge0}$ is a L\'{e}vy process independent of $(Y^n_t)_{t\ge0}$.
 For any $t>0$, set $Y_t=(Y_t^{n})_{n\ge1}$ and
 $\overline{Z}_t=(\overline{Z_t^n})_{n\ge1}$.

  Let $X_1$ and $X_2$ be the processes with components \eqref{th2000} by replacing $Z=(Z_t^n)_{n\ge1}$ with $Y$ and $\overline{Z}$, respectively. Since \eqref{not2} holds, we know form Corollary \ref{thex1} that the processes $X_1$ and $X_2$ take values in $\Hh$, and both are Markovian.
 Let $(P_t^{X_1})_{t\ge0}$ and $(\overline{P_t})_{t\ge0}$ be the semigroups corresponding to $X_1$ and $X_2$, respectively. Then, by the independence of $(Y_t^n)_{t\ge0}$ and $(\overline{Z_t^n})_{t\ge0}$, we have
$$P_t=P_t^{X_1}\overline{P_t},\quad t\ge0.$$ Therefore, for any $x,z\in\Hh$, $t>0$ and $g\in B_b(\Hh)$,
$$\aligned |\nabla_zP_tg(x)|&=\limsup_{\varepsilon\to 0}\frac{1}{\varepsilon}\Big|P_tg(x+\varepsilon z)-P_tg(x)\Big|\\
&=\limsup_{\varepsilon\to 0}\frac{1}{\varepsilon}\Big|P_t^{X_1}\overline{P_t}g(x+\varepsilon z)-P_t^{X_1}\overline{P_t}g(x)\Big|\\
&\le\|\nabla P_t^{X_1}\|_\infty \|z\|\|\overline{P_t}g\|_{\Hh,\infty}\\
&\le \|\nabla P_t^{X_1}\|_\infty
\|z\|\|g\|_{\Hh,\infty},\endaligned$$ where the last inequality
follows from $\|\overline{P_t}\|_{\Hh\to \Hh}\le 1$ for any $t>0.$
 This implies that
$$\|\nabla P_t\|_\infty\le \|\nabla P_t^{X_1}\|_\infty,\quad t>0.$$
On the other hand, for any $t>0$ and $x,y\in\Hh$,
$$\aligned \|P_t(x,\cdot)-P_t(y,\cdot)\|_{\var}&=\sup_{\|g\|_{\Hh,\infty}\le 1}|P_tg(x)-P_tg(y)|\\
&=\sup_{\|g\|_{\Hh,\infty}\le
1}|P_t^{X_1}\overline{P_t}g(x)-P_t^{X_1}\overline{P_t}g(y)|\\
&\le \sup_{\|g\|_{\Hh,\infty}\le
1}|P_t^{X_1}g(x)-P_t^{X_1}g(y)|\\
&=\|P_t^{X_1}(x,\cdot)-P_t^{X_1}(y,\cdot)\|_{\var},
\endaligned$$ where the inequality above
also follows from $\|\overline{P_t}\|_{\Hh\to \Hh}\le 1$ for any
$t>0.$

Next, we turn to estimate $\|\nabla P_t^{X_1}\|_\infty$ and  $\|P_t^{X_1}(x,\cdot)-P_t^{X_1}(y,\cdot)\|_{\var}$. For any $n\ge1$, define
$$dX_t^{1,n}=-\gamma_nX_t^{1,n}dt+dY_t^n.$$
Then, the law of $X_t^{1,n}$ is an infinitely divisible probability distribution, and its characteristic exponent is
   $$\varphi_{n,t}^1(h):=-\log(\Ee(e^{ihX_t^{1,n}}))=\int_0^t\int\Big(1-e^{ihe^{-\gamma_ns}z}\Big)|z|^{-1}f_n(z^{-2})\,dz
   \,ds$$ for $t\ge0$ and $h\in\R,$ e.g.\ see \cite[Theorem 3.1]{SAT} and \cite[Proposition 2.1]{Mu}.
Note that
$$\aligned \varphi_{n,t}^1(h)=&\int_0^t\int\Big(1-\cos\big({he^{-\gamma_ns}z}\big)\Big)|z|^{-1}f_n(z^{-2})\,dz
   \,ds,\\
   =&\int_0^t\int\Big(1-\cos\big|{he^{-\gamma_ns}z}\big|\Big)|z|^{-1}f_n(z^{-2})\,dz
   \,ds,\\
   =&2\int_0^t\int_0^\infty(1-\cos u)u^{-1}f_n(e^{-2\gamma_ns}u^{-2}h^2)\,du\,ds.
\endaligned$$
Set $$\phi^1_{n,t}(r)=2\int_0^t\int_0^\infty(1-\cos
u)u^{-1}f_n(e^{-2\gamma_ns}u^{-2}r)\,du\,ds,\quad r>0.$$ Since, for
any $n\ge1$ and $u,s>0$, the function $r\mapsto(1-\cos
u)u^{-1}f_n(e^{-2\gamma_ns}u^{-2}r)$ is a Bernstein function, it
follows from \cite[Corollary 3.7]{SSZ} that $\phi^1_{n,t}$ is also a
Bernstein function. Now, according to Proposition
\ref{thsub}, for any $t>0$ and $x,y\in \Hh$,
\begin{equation*}\label{thsub1} \|\nabla P_t^{X_1}\|_\infty\le \sqrt{2\sup_{k\ge1}\bigg(e^{-2\gamma_kt}\int_0^\infty
e^{-\phi^1_{k,t}(r)}\,dr\bigg)},\end{equation*} and
\begin{equation*}\label{thsub122} \|P_t^{X_1}(x,\cdot)-P_t^{X_1}(y,\cdot)\|_{\var}\le 2 \sqrt{2\sup_{k\ge1}\bigg(e^{-2\gamma_kt}\int_0^\infty
e^{-\phi^1_{k,t}(r)}\,dr\bigg)}\|x-y\|.\end{equation*}
Furthermore, by \eqref{psps1},
$$\aligned \phi^1_{n,t}(r)\ge&(\cos1)\int_0^t\int_0^1uf_n(e^{-2\gamma_ns}u^{-2}r)\,du\,ds\\
=&\frac{\cos1}{2}\int_0^t\int_0^1f_n(e^{-2\gamma_ns}v^{-1}r)\,dv\,ds\\
\ge&\frac{\cos1}{2}\int_0^tf_n(e^{-2\gamma_ns}r)\,ds\\
\ge&\frac{\cos1}{2}\Big[\Big(tf_n(e^{-2\gamma_n t}r)\Big)\wedge\Big(\frac{1-e^{-\gamma_n t}}{2\gamma_n}f_n(e^{-\gamma_n t}r)\Big) \Big].\endaligned$$ Combining with all estimates above yields the required assertion.
\end{proof}

We close this section with some comments on Theorem \ref{th2}.

\begin{remark}
(1) When the right hand side of \eqref{th22} is finite for any $t>0$, the gradient estimate \eqref{th22}
implies the strong Feller property of the semigroup $(P_t)_{t\ge0}$,
i.e.\ for any $t>0$ and $g\in B_b(\Hh)$, $P_tg\in C_b(\Hh)$. Indeed,
in this case one can follow the proof above and show that under the
assumptions in Theorem \ref{th2}, for any $t>0$ and $g\in B_b(\Hh)$,
$P_tg\in C_b^\infty (\Hh)$, where $C_b^\infty(\Hh)$ is the space of
all infinite differentiable functions $g$ on $\Hh$ such that their
Fr\'{e}chet derivatives $D^ig$, for all $i\ge1$, are continuous and
bounded on $\Hh.$ See \cite{WJ3} for the recent study on this topic
in finite dimension setting.

(2) Suppose that $\int_0^\infty r^{-1}f_n(r^{-2})\,dr=\infty$ for any $n\ge1$. Then, under condition \eqref{th21} and by
applying \cite[Theorem 24.10 (ii)]{Sa}, we can follow the proof of
\cite[Theorem 3.3]{PZ2} to prove that the process $X$ with
components \eqref{th2000} is irreducible; that is, for any open
ball $D\subset\Hh$ and $t>0$, we have $\Pp(X_t^x\in D)>0$.
Furthermore, according to the Hasminkii theorem (see
\cite[Proposition 4.1.1]{PZ0}), for any $t>0$ and $x,y\in \Hh$,
the laws of $X_t^x$ and $X_t^y$ are equivalent, e.g.\ see
\cite[Theorem 4.12 and Remark 4.16]{PZ1}.
\end{remark}

\section{Exponential Ergodicity of Linear Evolution  Equations with Cylindrical
L\'{e}vy noise}\label{section5} We begin with the

\begin{proof}[Proof of Theorem $\ref{ergociff}$] (1) First, we mention that the assumption (i) implies that the following two conditions are satisfied:
\begin{itemize}
\item[(iii)] $$\sum_{n=1}^\infty \int_0^\infty\bigg(e^{-2\gamma_n s}\int_{\{|z|\le e^{\gamma_ns}\}}z^2\,\nu_n(dz)+ \int_{\{|z|>e^{\gamma_ns}\}}\,\nu_n(dz)\bigg)\,ds<\infty;$$

\item[(iv)]$$\sum_{n=1}^\infty \int_0^\infty e^{-\alpha\gamma_ns}\int_{\{|z|>e^{\gamma_ns}\}}|z|^\alpha\,\nu_n(dz)\,ds<\infty.$$
\end{itemize}

According to (iii) and Corollary \ref{thex1}, the process
$X=(X_t^x)_{t\ge 0}$ takes values in $\Hh$, and it is Markovian.
Without loss of generality, we can assume that $C_t<\infty$ for any
$t>0$. Then, by Theorem \ref{th2} and (ii), for any $t>0$ and
$x,y\in\Hh$,
\begin{equation}\label{ergocthereom}\aligned\|P_t(x,\cdot)-P_t(y,\cdot)\|_{\var}
&\le 2C_t\|x-y\|,\endaligned\end{equation}
where $$C_t=\sqrt{2\sup_{k\ge1}\bigg[\bigg(\int_0^\infty e^{-\frac{\cos1}{2}tf_k(r)}\,dr\bigg)\wedge\bigg(e^{-\gamma_kt}\int_0^\infty
e^{-\frac{(\cos1)(1-e^{-2\gamma_kt})}{4\gamma_k}f_k(r)}\,dr\bigg)\bigg]}.$$
(2) Clearly, (iv) gives us that
\begin{itemize}
\item[(v)] for any $n\ge1$, $$\int_{\{|z|>1\}}|z|^\alpha\,\nu_n(dz)<\infty;$$ in particular, $$\int_{\{|z|>1\}}\log(1+|z|)\,\nu_n(dz)<\infty.$$
\end{itemize}
Due to the assumptions (iii) and (v), one
can follow the proof of \cite[Proposition 2.11]{PZ2} to get that
there exists an invariant probability measure $\mu$ for the process $X$.

Indeed,
by the assumption (v) and the fact that $\gamma_n>0$, according to \cite[Theorem 4.1]{SAT} or \cite[Theorem 17.5]{Sa}, for each $n\ge1$, one dimensional Ornstein-Uhlenbeck process $(X_t^n)_{t\ge0}$ has an invariant probability measure $\mu_n$, which is the law of the random variable
$$\int_0^\infty e^{-\gamma_ns}\,dZ_s^n.$$
Therefore, the product measure $$\mu=\prod_{n=1}^\infty \mu_n$$ on $\R^\N$ is an invariant measure of $X=(X_t^x)$. To prove that $\mu$ is a probability measure on $\Hh$, it remains to show that $\mu(\Hh)=1$. Let $\xi=(\xi_n)$ be a random variable on $\R^\N$ with
$$\xi_n=\int_0^\infty e^{-\gamma_ns}\,dZ_s^n,\quad n\ge1.$$ According to Proposition \ref{thlevy}, $\xi$ takes values in $\Hh$ if and only of the assumption (iii) holds.

(3) We further infer that the process
$X$ is ergodic, i.e.\ the invariant measure $\mu$ is unique, and for any $x\in\Hh$,
$$\lim_{t\to\infty}\|P_t(x,\cdot)-\mu\|_{\var}=0.$$
Note that \eqref{ergocthereom} and $\lim_{t\to\infty}C_t=0$ imply that when $t\to \infty$,
$$\|P_t(x,\cdot)-P_t(y,\cdot)\|_{\var}$$ converges to zero locally uniformly for all $x,y\in \Hh$.
Let $(P_t)_{t\ge0}$ be the semigroup associated with the process $X$. Since $\mu P_t=\mu$ for any $t>0$,
$$\|P_t(x, \cdot)-\mu\|_{\var}\le \int\|P_t(x, \cdot)-P_t(y,\cdot)\|_{\var}\,\mu(dy)$$  holds for any $x\in \Hh$ and $t>0$. This along with the statement above yields the second required
assertion.

Let $\mu_1$ and $\mu_2$
be invariant measures for $(P_t)_{t\ge0}$. Then,
$$\|\mu_1-\mu_2\|_{\var}\le \int\|P_t(x,
\cdot)-P_t(y,\cdot)\|_{\var}\,\mu_1(dx)\,\mu_2(dy),$$ which
gives us that $\mu_1=\mu_2$, by
letting $t\to\infty.$ This proves the
uniqueness of invariant measure.

(4) To obtain the explicit estimates for the ergodicity of the process $X$, we shall further make full use of \eqref{ergocthereom}. For any $t, s>0$, $g\in B_b(\Hh)$ and $\alpha\in(0,1)$, by the Markov property,
$$\aligned |P_tg(x)-P_{t+s}g(x)|=&|\Ee (P_tg(x)-P_tg(X_s^x))|\\
\le &\Ee|P_tg(x)-P_tg(X_s^x)|\\
=&\Ee\bigg(\frac{|P_tg(x)-P_tg(X_s^x)|^\alpha}{\|X_s^x-x\|^\alpha} |P_tg(x)-P_tg(X_s^x)|^{1-\alpha}\|X_s^x-x\|^\alpha\bigg)\\
\le &(2C_t)^\alpha \|g\|_{\Hh,\infty}^\alpha \times  (2\|g\|_{\Hh,\infty})^{1-\alpha}\times\Ee\big( \|X_s^x-x\|^\alpha\big)\\
=&2C_t^\alpha\|g\|_{\Hh,\infty} \Ee\big(\|T_sx-x+Y_s\|^\alpha\big),
\endaligned$$ where $T_sx=(e^{-\gamma_ns}x_n)_{n\ge1}$ and $$Y_s=\bigg(\int_0^se^{-\gamma_n(s-u)}\,dZ_u^n\bigg)_{n\ge1}.$$

Since $\alpha\in (0,1]$, it holds that $$
\Ee\big(\|T_sx-x+Y_s\|^\alpha\big)\le\Ee\big((\|T_sx-x\|+\|Y_s\|)^\alpha\big)\le
\|T_sx-x\|^\alpha+\Ee\|Y_s\|^\alpha$$ and $$\|T_sx-x\|^\alpha\le
\|T_sx\|^\alpha+\|x\|^\alpha\le
(e^{-s\alpha\inf_{n\ge1}\gamma_n}+1)\|x\|^\alpha\le
2\|x\|^\alpha,$$ where we have used the fact that for any $a,b\ge0$ and $\alpha\in(0,1]$,
 $$(a+b)^\alpha\le a^\alpha+b^\alpha.$$ On the other hand, under the assumptions (iii) and (iv), one will prove below that $\Ee\|Y_s\|^\alpha$ is uniformly bounded for all $s>0$, i.e.\ \begin{equation}\label{ppp}\sup_{s>0}\Ee\|Y_s\|^\alpha<\infty.\end{equation} If \eqref{ppp} holds, then  $$\sup_{s>0}\Ee\big(\|T_sx-x+Y_s\|^\alpha\big)\le C(1+\|x\|^\alpha)$$holds for any $x\in \Hh$ and some constant $C>0.$ Combining with all the conclusions above, we get that
$$|P_tg(x)-P_{t+s}g(x)|\le C(1+\|x\|^\alpha)C_t^\alpha\|g\|_{\Hh,\infty}.$$ That is,
$$\|P_t(x,\cdot)-P_{t+s}(x,\cdot)\|_{\var}\le  C(1+\|x\|^\alpha)C_t^\alpha.$$ The proof is completed by letting $s\to \infty$ and noting that $\mu$ is the invariant measure of $(P_t)_{t\ge0}$.

(5) Next, we turn to prove \eqref{ppp}. For any $n\ge1$ and $t>0$,
set $$Y^n_t=\int_0^t e^{-\gamma_n(t-s)}\, dZ^n_s.$$ Since
$(Z_t^n)_{t\ge0}$ is a symmetric L\'{e}vy process, $Y_t^n$ is a
symmetric infinitely divisible random variable with L\'{e}vy measure
$\nu_{n,t}$ as follows:
$$\nu_{n,t}(D)= \int_0^t\nu_n(e^{-s\gamma_n}D)\,ds,\qquad D\in\mathscr{B}(\R\setminus\{0\}).$$
In the following, we see $Y_t^n$ as the position of some symmetric L\'{e}vy process with L\'{e}vy measure $\nu_{n,t}$ at time 1. According to the L\'{e}vy-It\^{o} decomposition, see \cite[Chapter 4]{Sa}, we know that
$$Y_t^n=\int_0^1\int_{\{|z|\le 1\}}z \,\widetilde{N}_{n,t}(dz,ds)+\int_0^1\int_{\{|z|> 1\}}z \,N_{n,t}(dz,ds)=:Y_t^{n,1}+Y_t^{n,2},$$
where $\widetilde{N}_{n,t}(dz,ds)$ is the compensated Poisson measure, i.e.\ $$\widetilde{N}_{n,t}(dz,ds)=N_{n,t}(dz,ds)-\nu_{n,t}(dz)\,ds.$$ Set $Y_t^1=\big(Y_t^{n,1})_{n\ge1}$ and $Y_t^2=\big(Y_t^{n,2})_{n\ge1}$. Since $\widetilde{N}_{n,t}(dz,ds)$ is a square integrable martingale measure,
for any $t>0$, by the Cauchy-Schwarz inequality,
$$\aligned \Ee\|Y_t^1\|^\alpha\le&\Ee(\|Y_t^1\|^2)^{\alpha/2}\\
=&\Big(\sum_{n=1}^\infty\Ee |Y_t^{n,1}|^2\Big)^{\alpha/2} \\
=&\Big(\sum_{n=1}^\infty\int_0^1\,ds\int z^2\,\nu_{n,t}(dz)\Big)^{\alpha/2}\\
=&\Big(\sum_{n=1}^\infty\int_0^te^{-2\gamma_ns}\,\int_{\{|z|\le
e^{\gamma_ns}\}}z^2\,\nu_n(dz)\,ds\Big)^{\alpha/2}.\endaligned$$ On
the other hand, $Y_t^{n,2}$ is a compound Poisson random variable
with intensity $\nu_{n,t}(\{z\in \R,|z|>1\}).$ As the same argument
as that for $Y_t^n$, we regard $Y_t^{n,2}$ as the position of some compound
Poisson process $(E_{n,t}(s))_{s\ge0}$ with L\'{e}vy measure
$\I_{\{|z|>1\}}(z)\,\nu_{n,t}(dz)$ at time 1. According to the
It\^{o} formula for compound Poisson process, and the facts that
$Y_0^{n,2}=0$ and $(a+b)^\alpha\le a^\alpha+b^\alpha$ for any
$a,b\ge0$ and $\alpha\in(0,1]$, we find
$$\aligned \Ee|Y_t^{n,2}|^\alpha=& \Ee\int_0^1 \int_{\{|z|> 1\}} \Big(|E_{n,t}(s)+z|^\alpha-|E_{n,t}(s)|^\alpha\Big)\,\nu_{n,t}(dz)\,ds\\
\le& \Ee\int_0^1 \int_{\{|z|> 1\}} \Big((|E_{n,t}(s)|+|z|)^\alpha-|E_{n,t}(s)|^\alpha\Big)\,\nu_{n,t}(dz)\,ds\\
\le &\int_0^1\,ds\int_{\{|z|>1\}} |z|^\alpha\,\nu_{n,t}(dz)\\
 =&\int_0^t e^{-\alpha\gamma_ns}\int_{\{|z|>e^{\gamma_ns}\}}|z|^\alpha\,\nu_n(dz)\,ds. \endaligned$$ Thus, we get that
$$\Ee \|Y_t^2\|^\alpha\le\sum_{n=1}^\infty\Ee|Y_t^{n,2}|^\alpha\le \sum_{n=1}^\infty\int_0^t e^{-\alpha\gamma_ns}\int_{\{|z|>e^{\gamma_ns}\}}|z|^\alpha\,\nu_n(dz)\,ds  .$$
Combining with all the conclusions above with the assumptions (iii) and (iv), we get that
$$\aligned\sup_{t>0}\Ee \|Y_t\|^\alpha&\le \Big(\sum_{n=1}^\infty\int_0^\infty e^{-2\gamma_ns}\,\int_{\{|z|\le e^{\gamma_ns}\}}z^2\,\nu_n(dz)\,ds\Big)^{\alpha/2}\\
&\quad+\sum_{n=1}^\infty\int_0^\infty e^{-\alpha\gamma_ns}\int_{\{|z|>e^{\gamma_ns}\}}|z|^\alpha\,\nu_n(dz)\,ds\\
&<\infty.\endaligned$$The required assertion \eqref{ppp} follows.
\end{proof}

To conclude this section and to furthermore illustrate the power of
Theorem \ref{ergociff}, we take the following example, which has
been studied in \cite[Example 2.7]{EXZ1} and \cite[Example
2.9]{EXZ2}.

\begin{example}\label{improv} Consider the following stochastic linear equation on $D=[0,\pi]^d$ with the Dirichlet boundary condition
\begin{equation*}\label{levycutoff}
    \begin{cases}
        dX(t,\xi)=\Delta X(t,\xi)\,dt+dZ_t(\xi),\\
        X(0,\xi)=X(\xi),\\
        X(t,\xi)=0,\quad \xi\in \partial D
    \end{cases}
\end{equation*}
 where $(Z_t)_{t\ge0}$ is defined below. The Laplace operator $-\Delta$ with the Dirichlet
 boundary condition on $D$ has the following eigenfunctions
 $$e_n=\bigg(\frac{2}{\pi}\bigg)^{d/2}\sin (n_1\xi)\,\cdots \sin (n_d\xi_d),\quad n\in \N^d,\xi\in D,$$
 and eigenvalues $$\gamma_n=|n|^2=n_1^2+n_2^2+\cdots+n_d^2,\quad n\in \N^d.$$
 We study the dymamics defined above in the Hilbert space $\Hh=L^2(D)$ with the orthonormal basis $\{e_n\}$. Let us assume that $Z=(Z_t)$ is a cylindrical $\alpha$-stable noise written in the form
$$Z_t=\sum_{n\in \N^d}|n|^\beta Z_t^ne_n,\quad t\ge0,$$ where $(Z_t^n)_{n\ge1}$ are i.i.d.\
symmetric $\alpha$-stable processes with $\alpha\in(0,2)$ and $\beta$ is a real number.
According to Theorem \ref{ergociff}, if
$$\frac{d}{\alpha}<\frac{2}{\alpha}-\beta,$$ then the system indeed
takes values in $\Hh$, and it is exponentially ergodic. The
condition above implies that by choosing $\beta$ negative enough the
system can be exponentially ergodic for all $d\ge1$, which also
extends \cite[Example 2.9]{EXZ2}.\end{example}

\begin{ack} Financial support through National Natural Science Foundation of China (No.\ 11201073), the Program for New Century Excellent Talents in Universities
of Fujian (No.\ JA12053) and the Program for Nonlinear Analysis and Its
Applications (No.\ IRTL1206) is gratefully acknowledged. The author
would like to thank the referee for careful reading and valuable
comments and suggestions.
\end{ack}

\end{document}